\documentclass[12pt]{article}  
\usepackage{amssymb}
\usepackage{amsmath,amsfonts,amsthm}
\usepackage{mathtools}

\newtheorem{theorem}{Theorem}[section]
\newtheorem{lemma}[theorem]{Lemma}
\newtheorem{proposition}[theorem]{Proposition}
\newtheorem{corollary}[theorem]{Corollary}

\newtheorem{example}[theorem]{Example}

\usepackage{tikz-cd}
\usepackage[pdf]{pstricks}
\usepackage{chemfig}
\usepackage{chemmacros}
\usetikzlibrary{cd}
\usepackage[pagewise]{lineno}

\usepackage{fullpage}

\begin{document}

\numberwithin{equation}{section}

\title{On noncompact warped product Ricci solitons}

\author{Valter Borges
}


\maketitle{}

\begin{abstract}
	The goal of this article is to investigate complete noncompact warped product gradient Ricci solitons. Nonexistence results, estimates for the warping function and for its gradient are proven. When the soliton is steady or expanding these nonexistence results generalize to a broader context certain pde estimates and rigidity obtained when studying warped product Einstein manifolds. When the soliton is shrinking, it is presented a nonexistence theorem with no counterpart in the Einstein case, which is proved using properties of the first eigenvalue of a weighted Laplacian.
\end{abstract}

\vspace{0.2cm} 

 \section{Introduction and Main Results}
 
 A {\it Ricci soliton} $(N,g_{N},X,\lambda)$ is a Riemannian manifold $(N,g_{N})$ with a vector field $X$ and a constant $\lambda\in\mathbb{R}$ satisfying the equation
 \begin{eqnarray*}
 	\displaystyle Ric_{N}+\frac{1}{2}\mathfrak{L}_{X}g_{N}=\lambda g_{N}.
 \end{eqnarray*}
If $\lambda$ is positive, zero or negative, the soliton is called \textit{shrinking}, \textit{steady} or \textit{expanding}, respectively. If the vector field $X$ is the gradient of some function $f:N\rightarrow\mathbb{R}$, then the soliton is called a {\it gradient Ricci soliton}. In this case, it is denoted by  $(N,g_{N},f,\lambda)$ and the defining equation becomes
 \begin{eqnarray}\label{eqriccisoliton}
 Ric_{N}+\nabla_{N}\nabla_{N} f=\lambda g_{N},
 \end{eqnarray}
where $\nabla_{N}\nabla_{N} f$ is the Hessian of $f$ with respect to the metric $g_{N}$. The function $f$ is called the \textit{potential function}.

In this paper we investigate Ricci solitons $N^{n+m}$ which globally have the geometry of a complete warped product $M^n\times_{h}F^m$. Here, $h:M\rightarrow\mathbb{R}$ is an everywhere positive smooth function. When $h$ is not constant and $M$ is complete, it was shown in \cite[Corollary 2.2]{bo-te} that $f$ is lifted from $M$. This reduces the function theoretic analysis of the Ricci soliton equation on complete warped products only to the base $M$. In this case, $(\ref{eqriccisoliton})$ is equivalent to
\begin{equation}\label{system_withoutdependence}
	\begin{array}[pos]{lll}
		Ric+\nabla\nabla f-mh^{-1}\nabla\nabla h=\lambda g,\\\noalign{\smallskip}
		\lambda h^{2}=h(\nabla h)f-(m-1)|\nabla h|^{2}-h\Delta h+\mu,\\\noalign{\smallskip}
		Ric_{F}=\mu g_{F},
	\end{array}
\end{equation}
where all functions and tensors in the first two equations are lifted from $M$. This equivalence was originally obtained in \cite[Theorem 3]{fefrego}, under the additional assumption that $f$ is lifted from the base.

Warped product gradient Ricci solitons arise naturally in classification problems, such as when assuming that the soliton has harmonic Weyl tensor \cite{kim,li}. Furthermore, there are important examples of Ricci solitons, such as the Bryant soliton \cite{bryant}, constructed on warped products. Ricci solitons on manifolds with this type of geometry were investigated in \cite{bo-te,fefrego,gomari,lemespina}.

When $f$ is constant in $(\ref{eqriccisoliton})$, the corresponding equation describes Einstein manifolds, once
\begin{eqnarray}\label{eqeinstein}
	Ric_{N}=\lambda g_{N}.
\end{eqnarray}
Einstein manifolds on warped products $N=M\times_{h}F$ have been wildly studied in the last years \cite{barros,case,case2,kimkim,qian,rimoldi,wang,wang2}. Some motivation and examples can be found in \cite{besse}. 
When $N=M\times_{h}F$ and $v=-m\ln{h}$, it was shown in \cite{kimkim} that $(\ref{eqeinstein})$ (or $(\ref{system_withoutdependence})$ for $f$ constant) is equivalent to $F^m$ being Einstein and
\begin{equation}\label{mquasei}
		Ric+\nabla\nabla v=\lambda g+\frac{1}{m}d v\otimes dv.
\end{equation}
More precisely, starting from $(\ref{mquasei})$, they deduced the existence of $\mu\in\mathbb{R}$ so that $\Delta v-|\nabla v|^2=m\lambda-m\mu e^{\frac{2}{m}v}$ and thus, considered an Einstein manifold $F$ so that $Ric_{F}=\mu g_{F}$, from which they were able to build the Einstein warped product $M\times_{h}F$. Given $m\in(0,+\infty]$, $(M^n,g,v,\lambda)$ is called an {\it $m$-quasi-Einstein manifold} if equation $(\ref{mquasei})$ is satisfied. Observe that $m=+\infty$ drives $(\ref{mquasei})$ back to the gradient Ricci soliton equation with potential function $v$. Below we describe some results about these metrics that motivates the present work.

In \cite[Theorem 1]{kimkim}, Kim and Kim proved that a compact  $m$-quasi-Einstein manifold with $\lambda>0$ and $m\in(0,+\infty)$ must also have $\mu>0$. On the other hand, Qian \cite[Theorem 5]{qian} proved that $M$ is compact, if $\lambda>0$ and $m\in(0,+\infty)$. Recently, Wang \cite[Theorem 2.3]{wang} proved that if $M$ is compact and $v$ is not constant, then $\lambda>0$ and $\mu>0$. When $M$ is noncompact, there has also been some development. Notice that by Qian's result mentioned above we must have $\lambda\leq0$ in this case. When $\lambda=0$, Case \cite[Theorem 1.2]{case} and Wang \cite[Theorem 3.3]{wang} showed independently that one necessarily has $\mu>0$, unless $v$ is constant. When $\lambda<0$, some kind of boundedness is required to obtain nonexistence results. In \cite[Theorem 3.6]{wang} it was proved that $v$ is constant, provided $\sup_{M}|\nabla v|^2<-m^2\lambda/(n+m)$. Barros, Batista and Ribeiro Jr, on the other hand, proved in \cite[Theorem 5]{barros} that if $\mu<0$, then $v$ is constant, provided $v\leq2\mu/\lambda$ all over $M$. Aside from rigidity results, there are also some important estimates that we want to mention. Still under the assumption that $\lambda<0$, when $\mu\leq0$, Wang proved in \cite[Theorem 3.2]{wang} that $|\nabla v|^2\leq-m\lambda$. If in addition $\mu<0$, he proved in \cite[Theorem 5.3]{wang2} that $v\geq\m\ln\sqrt{\lambda/\mu}$. There are also interesting estimates on the volume growth of geodesic balls and on the scalar curvature that we do not mention. The interested reader may consult \cite{barros,case2,wang,wang2}.

In what follows we describe some results concerning complete warped product gradient Ricci solitons, related to our work. When the base is compact and $m\geq2$, Feitosa, Freitas and Gomes \cite{fefrego} generalized \cite[Theorem 1]{kimkim} and \cite[Theorem 2.3]{wang} showing that $\lambda>0$ and $\mu>0$, if $h$ is not constant. When $M$ is noncompact, $\lambda\leq0$ and $m\geq2$, they also proved rigidity, under the additional assumption that both the maximum and the minimum of the warping function are attained \cite[Theorem 1]{fefrego}. It is worthy noting that when $\lambda\neq0$ and $m=1$, a solution of the first two equations of $(\ref{system_withoutdependence})$ do not give rise to a warped product Ricci soliton (see \cite[Remark 2.3]{case2}). Taking into account this remark, their result reads as.
\begin{theorem}[\cite{fefrego}]\label{compactbasis}
	Let $M^n\times_{h}F^m$ be a gradient Ricci soliton and $R_{F}$ the scalar curvature of $F^m$. Suppose that $M^n$ is compact, fix $q_{0}\in F$ and let $\mu$ be defined by $m\mu=R_{F}(q_{0})$. The following are true:
	\begin{enumerate}
		\item $m\geq2$, $\lambda>0$, $\mu>0$ and $F$ is Einstein.
		\item\label{2estimate} Suppose $\lambda>0$ and $\mu>0$. If either $\displaystyle h\leq\sqrt{\frac{\mu}{\lambda}}$ or $\displaystyle h\geq\sqrt{\frac{\mu}{\lambda}}$ on $M$, then $h$ is constant.
	\end{enumerate}
\end{theorem}

The proof of item \ref{2estimate} of the theorem above is implicit in the proof of \cite[Theorem 1]{fefrego}. In Section \ref{prelimi}, we highlight its proof.

A natural question is whether Theorem \ref{compactbasis} holds for complete noncompact warped product gradient Ricci solitons. The answer is partially positive, in the sense that assuming $\lambda>0$ we obtain similar conclusions.

\begin{theorem}\label{noncompacshrinking}
	Let $M^n\times_{h}F^m$ be a complete gradient shrinking Ricci soliton and $R_{F}$ the scalar curvature of $F^m$. Suppose that $M$ is noncompact, fix $q_{0}\in F$ and let $\mu$ be defined by $m\mu=R_{F}(q_{0})$. The following are true:
	\begin{enumerate}
		\item\label{item01} If $h$ is not constant, then $\mu>0$.
		\item\label{item02} If $\displaystyle h\leq\sqrt{\frac{\mu}{\lambda}}$, then $h$ is constant.
	\end{enumerate}
\end{theorem}

Item \ref{item02} in the theorem above seems to be the first restriction in the literature to the existence of warped product Ricci solitons when $\lambda>0$, $\mu>0$ and $M$ is noncompact, and has no counterpart in the warped product Einstein setting.

In order to prove item \ref{item01} of Theorem \ref{noncompacshrinking}, motivated by \cite{wang}, we localize the equation satisfied by $|\nabla (m\ln{h})|^2$ and apply the maximum principle. To obtain the rigidity of item \ref{item02}, we work with the first eigenvalue $\lambda_{1}(-\Delta_{f-m\ln h})$ of the weighted Laplacian
\begin{align*}
	\Delta_{f-m\ln h}u=\Delta u-\left\langle\nabla(f-m\ln h),\nabla u\right\rangle.
\end{align*}
It turns out that such an eigenvalue vanishes when $M$ is noncompact and $\lambda>0$, and this is essential in our proof. This strategy is inspired by nonexistence results due to Fujita \cite{fujita}.

Turning to noncompact steady warped product gradient Ricci solitons we prove the following.

\begin{theorem}\label{noncompacsteady}
	Let $M^n\times_{h}F^m$ be a complete gradient steady Ricci soliton and $R_{F}$ the scalar curvature of $F^m$. Suppose that $M$ is noncompact, fix $q_{0}\in F$ and let $\mu$ be defined by $m\mu=R_{F}(q_{0})$. The following are true:
	\begin{enumerate}
		\item\label{item1} If $h$ is not constant, then $\mu>0$.
		\item\label{item2} If $\displaystyle\sup_{M}h<+\infty$, then $h$ is constant.
	\end{enumerate}
\end{theorem}

Item \ref{item1} of Theorem \ref{noncompacsteady} is an extension of \cite[Theorem 1.2]{case} and \cite[Theorem 3.3]{wang} to warped product gradient Ricci solitons. We do not know if item \ref{item2} was already known in the case of warped products Einstein manifolds.

Theorem \ref{noncompacshrinking} and Theorem \ref{noncompacsteady} assert that rigidity or nonexistence is obtained when $\lambda\geq0$, provided $\mu\leq0$. The same does not hold for expanding warped product gradient Ricci solitons, as there are examples \cite[Corollary 2]{fefrego} (see also Section \ref{prelimi} below) which are complete, have nonconstant warping function, nonconstant potential function and $\mu\leq0$. In this case, we are able to prove the following.

\begin{theorem}\label{noncompacexpanding}
	Let $M^n\times_{h}F^m$ be a complete gradient expanding Ricci soliton and $R_{F}$ the scalar curvature of $F^m$. Suppose that $M$ is noncompact, fix $q_{0}\in F$ and let $\mu$ be defined by $m\mu=R_{F}(q_{0})$. The following are true:
		\begin{enumerate}
			\item\label{negitem1} If $\mu<0$, then $\displaystyle h\geq\sqrt{\frac{\mu}{\lambda}}$.
			\item\label{negitem2} If $\mu\leq0$, then $\displaystyle |\nabla\ln{h}|^2\leq-\frac{\lambda}{m}$.
			\item\label{negitem3} If $\mu<0$ and $\displaystyle\sup_{M}h<+\infty$, then $\displaystyle |\nabla\ln{h}|^2\leq-\frac{\lambda}{m}+\frac{2\mu}{m\left(\sup_{M}h\right)^2}$.
			\item\label{negitem4} If $\mu\geq0$ and $\displaystyle\sup_{M}h<+\infty$, then $h$ is constant.
		\end{enumerate}
\end{theorem}
We observe that items \ref{negitem1} and \ref{negitem2} extend \cite[Theorem 3.2]{wang} and \cite[Theorem 5.3]{wang2}, respectively, to warped product gradient Ricci solitons. Furthermore, the examples of \cite[Corollary 2]{fefrego} show that the estimate of item \ref{negitem2} is sharp. Again, we do not know if item \ref{negitem4} was already known in the case of warped products Einstein manifolds. Item \ref{negitem3}, on the other hand, has an immediate consequence, state below, which generalizes \cite[Theorem 5]{barros}.

\begin{corollary}
	Let $M^n\times_{h}F^m$ be a gradient expanding Ricci soliton and $R_{F}$ the scalar curvature of $F^m$. Suppose that $M$ is complete and noncompact, fix $p_{0}\in M$ and let $\mu$ be defined by $m\mu=R_{F}(q_{0})$. If $\mu<0$ and $\displaystyle h\leq\sqrt{\frac{2\mu}{\lambda}}$, then $h$ is constant.
\end{corollary}

Other rigidity results concerning both Einstein manifolds and gradient Ricci solitons on warped product manifolds can be found in \cite{rimoldi} and \cite{gomari}, respectively, where the authors conclude nonexistence by assuming integrability conditions on the warping and potential functions. Lastly, we want to remark that Theorem \ref{noncompacsteady} and Theorem \ref{noncompacexpanding} above improve Theorem 1.2 and Theorem 1.5 of \cite{gomari}, respectively.

In order to prove the theorems above, it will be necessary to localize some equations. One of the requirements when using the distance function to perform cutoffs, is that one needs to have available some type of Laplace comparison theorem. This is the case for $m$-quasi-Einstein manifolds \cite[Lemma 3.1]{wang}. However, when one is working with the base of a gradient warped product Ricci soliton, a similar result is not known to hold, and one should get on with in a different way. To overcome this difficulty, we proceed as in Perelman \cite[Lemma 8.3]{perelman} (see also \cite[Lemma 18.6]{book3}) and as in \cite[Theorem 27.2]{book4}.

This paper is organized as follows. In Section \ref{prelimi}, we recall a useful change of dependent coordinates, first performed in \cite{gomari}, collect some useful equations and prove Theorem \ref{compactbasis}. In Section \ref{proofmainTh}, we present the proofs of our main theorems, arranged in the following way: in Subsection \ref{localize}, we localize certain elliptic equations and apply the maximum principle to get some information when $\mu\leq0$; in Subsection \ref{fujita'sargument}, we explore the first eigenvalue of a weighted Laplacian when $\lambda>0$ and $\mu>0$; in Subsection \ref{wmpatinf} we apply the weak Maximum Principle at Infinity to the case where $\lambda\leq0$; In Subsection \ref{howtoprove}, we prove Theorems \ref{noncompacshrinking}, \ref{noncompacsteady} and \ref{noncompacexpanding}.

\section{Preliminaries}\label{prelimi}

Let $(M^{n},g)$ and $(F^{m},g_{F})$ be Riemannian manifolds, $\pi:M\times F\rightarrow M$ and $\sigma:M\times F\rightarrow F$ the canonical projections and $h:M\rightarrow\mathbb{R}$ a positive smooth function. We consider on the the product manifold $N^{n+m}=M^{n}\times F^{m}$ the metric $g_{N}=g+h^{2}g_{F}$, defined by
\begin{equation*}
	g_{N}=\pi^{*}g+(h\circ\pi)^{2}\sigma^{*}g_{F}.
\end{equation*}
We denote this Riemannian manifold by $M\times_{h}F$ and call it the \textit{warped product} between $M$ and $F$. We call $M$ the \textit{base}, $F$ the \textit{fiber} and $h$ the \textit{warping function} of the warped product. For more information on the geometry of warped products, see \cite{bishop,oneill}.

Given a warped product $M\times_{h}F$, we consider a potential function $f:M\times F\rightarrow\mathbb{R}$ and a constant $\lambda$ satisfying $(\ref{eqriccisoliton})$. As we already mentioned in the introduction, assuming $M$ complete and $h$ nonconstant imply that $f$ is lifted from the base \cite[Corollary 2.2]{bo-te}. Furthermore, using the properties of a warped product manifold \cite{bishop,oneill}, one proves that $(\ref{eqriccisoliton})$ and $(\ref{system_withoutdependence})$ are equivalent. Before we proceed with the results of this section, consider the following example.
\begin{example}[Corollary 2 of \cite{fefrego}] 
	Consider $M=\mathbb{R}^n$ with coordinates $x=(y,x_{n})=(x_{1}\ldots,x_{n-1},x_{n})$, constants $c_{1},c_{2}\geq0$ with $c_{1}^2+c_{2}^2\neq0$, $m\in(0,+\infty)$ and $\lambda<0$. Define the functions $f(x)=\frac{\lambda}{2}|y|^2$ and $h(x)=c_{1}e^{\sqrt{-\frac{\lambda}{m}}x_{n}}+c_{2}e^{-\sqrt{-\frac{\lambda}{m}}x_{n}}$,
	consider the constant $\mu$ defined by $m\mu=4c_{1}c_{2}(m-1)\lambda$, and a complete Einstein manifold $F^{m}$ satisfying $Ric_{F}=\mu g_{F}$. A simple computation shows that equations $(\ref{system_withoutdependence})$ are satisfied. Consequently, $\mathbb{R}^{n}\times_{h}F^{m}$ is a complete warped product Ricci soliton with potential function $f$, $\lambda<0$, $\mu\leq0$ and noncompact base. Furthermore:
	\begin{enumerate}
		\item Taking $c_{1}=0$ and $c_{2}>0$, we have $\displaystyle|\nabla\ln{h}|^2=-\frac{\lambda}{m}$. This shows that the estimate of item \ref{negitem2} in Theorem \ref{noncompacexpanding} is sharp.
		\item Taking $m\in(1,+\infty)$, we have $\displaystyle h(x)\geq h(x_{0})=\sqrt{\frac{m\mu}{(m-1)\lambda}}>\sqrt{\frac{\mu}{\lambda}}$,  where $x_{0}$ is the minimum of $h$.
	\end{enumerate}
\end{example}

For our purposes we consider the following change of dependent variables, originally performed in \cite{gomari}.

\begin{proposition}[\cite{gomari}]\label{gomariprop}
	Let $M\times F$ be a product manifold, $\lambda,\ \mu\in\mathbb{R}$ and $f,\ h\in C^{\infty}(M)$, with $h$ positive. Then $M\times_{h}F$ is a complete warped product Ricci soliton with potential function $f$ and soliton constant $\lambda$ if and only if the functions
	\begin{equation}\label{change}
		\begin{array}[pos]{lll}
			v=-m\ln h\ \ \ \ \text{and}\ \ \ \ \varphi=f-m\ln h,
		\end{array}
	\end{equation}
	satisfy the equations
	\begin{align}
		&Ric_{\varphi}=\lambda g+\frac{1}{m}d v\otimes d v\label{eiteiS},\\
		&-\Delta_{\varphi}v=-m\lambda+m\mu e^{\frac{2}{m}v}\label{neweq2},
	\end{align}
	where $Ric_{\varphi}=Ric+\nabla\nabla\varphi$ is the Bakry-Emery Ricci tensor and $\Delta_{\varphi}$, the $\varphi$-Laplacian, is defined for a function $u\in C^{2}(M)$ by $\Delta_{\varphi}u=\Delta u-\langle\nabla\varphi,\nabla u\rangle$.
\end{proposition}

The proof of the proposition above is a straightforward computation. The point of highlighting it here is that these equations are similar to those characterizing a warped product Einstein manifold $M\times_{e^{-\frac{\varphi}{m}}}F$ with Einstein constant $\lambda$. These last equations are
\begin{align}\label{eiteiE}
	Ric_{\varphi}=\lambda g+\frac{1}{m}d\varphi\otimes d\varphi\ \ \ \text{and}\ \ \ -\Delta_{\varphi}\varphi=-m\lambda+m\mu e^{\frac{2}{m}\varphi},
\end{align}
where the fiber $F$ is also an Einstein manifold with Einstein constant $\mu$ and dimension $m$. Proposition \ref{gomariprop} allows a unified investigation, where $v=\varphi$ if and only if the warped product Ricci soliton is actually a warped product Einstein manifold. However, it should be pointed out that while the scalar equation in $(\ref{eiteiE})$ follows from  the tensorial one, it is not known whether $(\ref{neweq2})$ follows from $(\ref{eiteiS})$ in general.

On a measure metric space $(M,g,e^{-\varphi}dV)$, one has the following well known Bochner formula relating the $\varphi$-Laplacian with the Bakry-Emery Ricci tensor $Ric_{\varphi}$,
\begin{align*}
	\Delta_{\varphi}|\nabla u|^2=2|\nabla\nabla u|^2+2\left\langle\nabla u,\nabla\Delta_{\varphi}u\right\rangle+2Ric_{\varphi}(\nabla u,\nabla u).
\end{align*}
This identity seems to have appeared first in \cite[Proposition 2.1]{setti}. Plugging into it equations $(\ref{eiteiS})$ and $(\ref{neweq2})$, we have the following
\begin{proposition}
	Let $M^n$ be a manifold for which there are functions $\varphi,v\in C^{\infty}(M)$ and constants $\lambda,\ \mu\in\mathbb{R}$ and $m\in(0,+\infty)$ satisfying $(\ref{eiteiS})$ and $(\ref{neweq2})$. Then,
	\begin{align}\label{gradeq}
		\Delta_{\varphi}|\nabla v|^2=2|\nabla\nabla v|^2+2(\lambda-2\mu e^{\frac{2}{m}v})|\nabla v|^2+\frac{2}{m}|\nabla v|^4.
	\end{align}
\end{proposition}

From now on, all theorems will be stated in terms of $\varphi$ and $v$, assuming that they satisfy equations $(\ref{eiteiS})$ and $(\ref{neweq2})$. We also note that as it is usual in the case of $m$-quasi-Einstein manifolds, the results do not require $m\in(0,+\infty)$ to be integer.

For the sake of completeness, we include the proof of Theorem \ref{compactbasis} here. The proof of item \ref{2estimate} is implicit in \cite{fefrego}, and we emphasize it here.

\begin{proof}[{\bf Proof of Theorem \ref{compactbasis}}]
	Integrating equation $(\ref{neweq2})$ with respect to $e^{-\varphi}dV$, we obtain
	\begin{align}\label{intident}
		\lambda\int_{M}e^{-\varphi}dV=\mu\int_{M}e^{\frac{2}{m}v}e^{-\varphi}dV.
	\end{align}
	
	According to the equation above, if $\lambda=0$, then we have $\mu=0$. Therefore, $(\ref{neweq2})$ becomes $\Delta_{\varphi}v=0$. Let $p_{+}\in M$ be a point where $|\nabla v|^2$ attains its maximum. From equation $(\ref{gradeq})$ we get $0\geq\Delta_{\varphi}|\nabla v|^2(p_{+})\geq\frac{2}{m}|\nabla v|^4(p_{+})$, from where $v$ is constant.
	
	If $\lambda<0$, equation $(\ref{intident})$ implies that $\mu<0$. Let $p_{1}$ and $p_{2}$ be points of $M$ where $v$ attains its minimum and maximum, respectively. Then $(\ref{neweq2})$ implies that $e^{\frac{2}{m}v(p_{2})}\leq\lambda\mu^{-1}\leq e^{\frac{2}{m}v(p_{1})}$, implying that $v$ is constant.
	
	Now consider the case where $\lambda>0$. When $m=1$, the only possibility is $\mu=0$ (see Remark 2.3 in \cite{case2}), which is impossible, in view of $(\ref{intident})$. Therefore, $m\geq2$. In this case, equation $(\ref{intident})$ implies that $\mu>0$. Now, let us assume that $\displaystyle h\leq\sqrt{\mu\lambda^{-1}}$. Using $(\ref{change})$ we conclude that $\Delta_{\varphi}u\leq0$ on $M$, and since $u$ attains its minimum, it follows from the Strong Maximum Principle that $v$ is constant. Similarly, if $\displaystyle h\geq\sqrt{\mu\lambda^{-1}}$, then $v$ is constant.
\end{proof}

\section{Proofs}\label{proofmainTh}

In this section we proof Theorem \ref{noncompacshrinking}, Theorem \ref{noncompacsteady} and Theorem \ref{noncompacexpanding}. This will be done in several steps along the next three subsections.
\subsection{Localizing and using the Maximum Principle}\label{localize}

We start this section with the following lemma, where we localize some equations.

\begin{lemma}\label{lema1}
	Let $M^n$ be a manifold for which there are functions $\varphi,v\in C^{\infty}(M)$ and constants $\lambda\in\mathbb{R}$, $\mu\leq0$ and $m\in(0,+\infty)$ satisfying $(\ref{eiteiS})$ and $(\ref{neweq2})$. Fix $O\in M$  and let $r:M\rightarrow[0,+\infty)$ be the the distance from $O$, defined as $r(x)=d(O,x)$. Now, given $b>0$, denote by $\eta:[0,+\infty)\rightarrow[0,1]$ the smooth function
	\begin{align}
		\eta(t)=
		\left\{ 
		\begin{array}[pos]{lll}
			1,\ t\in[0,1]\\\noalign{\smallskip}
			0,\ t\in[1+b,+\infty)
		\end{array}
		\right.
	\end{align}
	which is nonincreasing and satisfies
	\begin{align}\label{bound}
		\eta''-2\frac{(\eta')^{2}}{\eta}\geq -C,
	\end{align}
	for a universal constant $C>0$. For each $R\geq2$, consider the cutoff function $F_{R}:M\rightarrow[0,1]$, defined by
	\begin{align*}
		F_{R}(x)=\eta\left(\frac{r(x)}{R}\right).
	\end{align*}
	Therefore, the functions $G_{R}(x)=F_{R}(x)|\nabla v(x)|^2$ and $H_{R}(x)=F_{R}(x)e^{\frac{2}{m}v(x)}$, $x\in M$, satisfy
	\begin{align}
		&\Delta_{\varphi}G_{R}-2\left\langle\frac{\eta'\nabla r}{\eta R},\nabla G_{R}\right\rangle\geq\left[\frac{\eta''}{\eta R^2}-\frac{2(\eta')^{2}}{\eta^2R^{ 2}}+\frac{\eta'}{\eta R}\Delta_{\varphi}r+2\lambda\right]G_{R}+\frac{2}{m\eta}G_{R}^{2},\label{eqcutgrad}\\
		&\Delta_{\varphi}H_{R}-2\left\langle\frac{\eta'\nabla r}{\eta R},\nabla H_{R}\right\rangle\geq\left[\frac{\eta''}{\eta R^2}-\frac{2(\eta')^{2}}{\eta^2R^{ 2}}+\frac{\eta'}{\eta R}\Delta_{\varphi}r+2\lambda\right]H_{R}-\frac{2\mu}{\eta}H_{R}^{2},\label{eqcutexp}
	\end{align}
	at any point $x\in M$ so that $F_{R}(x)>0$.
\end{lemma}
\begin{proof}
	Let us compute the $\varphi$-Laplacian of $G_{R}(x)=F_{R}(x)|\nabla v(x)|^2$ and show how to obtain $(\ref{eqcutgrad})$. Using $(\ref{gradeq})$ we get
	\begin{align*}
		\Delta_{\varphi}G_{R}&=|\nabla v|^2\Delta_{\varphi}F_{R}+2\left\langle\nabla F_{R},\nabla|\nabla v|^2\right\rangle+F_{R}\Delta_{\varphi}|\nabla v|^2\\
		&\geq|\nabla v|^2\Delta_{\varphi}F_{R}+2\left\langle\nabla F_{R},\nabla|\nabla v|^2\right\rangle+F_{R}\left[2(\lambda-2\mu e^{\frac{2}{m}v})|\nabla v|^2+\frac{2}{m}|\nabla v|^4\right]\\
		&=\frac{\Delta_{\varphi}F_{R}}{F_{R}}G_{R}+2\left\langle\frac{\nabla F_{R}}{F_{R}},\nabla G_{R}\right\rangle-\frac{2|\nabla F_{R}|^2}{F_{R}^2}G_{R}+2(\lambda-2\mu e^{\frac{2}{m}v})G_{R}+\frac{2}{mF_{R}}G_{R}^2,
	\end{align*}
	from where
	\begin{align}\label{part1}
		\Delta_{\varphi}G_{R}-2\left\langle\frac{\nabla F_{R}}{F_{R}},\nabla G_{R}\right\rangle&\geq\left[\frac{\Delta_{\varphi}F_{R}}{F_{R}}-\frac{2|\nabla F_{R}|^2}{F_{R}^2}+2\lambda-4\mu e^{\frac{2}{m}v}\right]G_{R}+\frac{2}{mF_{R}}G_{R}^2.
	\end{align}
	Now, using $\mu\leq0$ and
	\begin{align*}
		\begin{split}
		&\nabla F_{R}=\frac{\eta'}{R}\nabla r,\\
		&\Delta_{\varphi}F_{R}=\frac{\eta''}{R^2}+\frac{\eta'}{R}\Delta_{\varphi}r,
		\end{split}
	\end{align*}
	we obtain $(\ref{eqcutgrad})$.
	
	Proceeding in a similar way we compute the $\varphi$-Laplacian of $H_{R}(x)=F_{R}(x)e^{\frac{2}{m}v(x)}$ and deduce $(\ref{eqcutexp})$. In this case we need to compute $\Delta_{\varphi}e^{\frac{2}{m}v}$ from $(\ref{neweq2})$. We will omit the details.
\end{proof}

When $\varphi=v$, that is, for $m$-quasi-Einstein manifolds, equations $(\ref{eqcutgrad})$ and $(\ref{eqcutexp})$ were investigated by Wang in \cite{wang,wang2}. There, in order to get rid of the Laplacian of the distance function in these equations, the authors used a Laplacian comparison theorem available for $m$-quasi-Einstein manifolds. Since at our conditions it is not known if a similar result holds, we use an estimate presented in \cite[Lemma 18.6]{book3}, first proved in Perelman \cite[Lemma 8.3]{perelman}.

\begin{lemma}[\cite{book3,perelman}]\label{lemaperel}
	Let $M^n$ be a manifold for which there is a function $\varphi\in C^{\infty}(M)$ and a constant $\lambda\in\mathbb{R}$ so that $Ric_{\varphi}\geq\lambda g$. Let $x\in M$ be so that $r(x)>1$, with $r(x)=d(O,x)$. Then,
	\begin{align*}
		(\Delta_{\varphi}r)(x)\leq n-1+|\nabla\varphi(O)|+\max_{\overline{B_{O}(1)}}|Ric|-\lambda r(x).
	\end{align*}
\end{lemma}

We can finally state and proof the main results of this subsection.
\begin{theorem}\label{gradbound}
	Let $M^n$ be a complete manifold for which there are functions $\varphi,v\in C^{\infty}(M)$ and constants $\lambda,\ \mu\in\mathbb{R}$ and $m\in(0,+\infty)$ satisfying $(\ref{eiteiS})$ and $(\ref{neweq2})$. Suppose that $\mu\leq0$. If $\lambda\geq0$, then $v$ is constant. If $\lambda<0$, then $|\nabla v|^2\leq-m\lambda$.
\end{theorem}
\begin{proof}
	Assume that $v$ is not constant, and let $R\geq2$ be large enough so that $G_{R}$ is not constant, and let $x_{R}\in M$ be its maximum. Thus $\Delta_{\varphi}G_{R}\leq0$ and $\nabla G_{R}=0$ at $x_{R}$. Once $G_{R}$ is nonnegative and nonconstant, $G_{R}(x_{R})>0$. Therefore, Lemma \ref{lema1} asserts that the following holds at $x_{R}$
		\begin{align}
			G_{R}(x_{R})&\leq-\frac{m\eta'}{2R}(\Delta_{\varphi}r)(x_{R})-\frac{m\eta''}{2R^2}+\frac{m(\eta')^{2}}{\eta R^{ 2}}-m\lambda\eta\label{line1}\\
			&\leq-\frac{m\eta'}{2R}(\Delta_{\varphi}r)(x_{R})+\frac{mC}{2R^{2}}-m\lambda\eta,\label{line2}
		\end{align}
	where in the second inequality we have used $(\ref{bound})$.
	
	We divide the proof in several cases, following the same rout as \cite[Theorem 27.2]{book4}. We also will use $C$ to denote a positive constant, which does not depend on $R$, and which may change from line to line.
	
	Suppose that $x_{R}\in B_{R}(O)$. In this case $F_{R}$ is constant equals to $1$ around $x_{R}$. Since in this case $\eta'=\eta''=0$, for any $x\in B_{O}(R)$ the inequality $(\ref{line1})$ gives,
	\begin{align}\label{lamneg}
		|\nabla v|^2(x)=G_{R}(x)\leq G_{R}(x_{R})\leq-m\lambda.
	\end{align}

	Now suppose that $x_{R}\notin B_{R}(O)$. Once $R\geq2$, Lemma \ref{lemaperel} allows us to further estimate $(\ref{line2})$ to obtain
	\begin{align}\label{tocarry}
		G_{R}(x_{R})\leq\frac{C}{R}+\frac{C}{R^{2}}-m\lambda\eta+\frac{m\lambda\eta'}{2R}r(x_{R}).
	\end{align}
	
	If $\lambda\geq0$, then $(\ref{lamneg})$ and the assumption that $v$ is not constant imply that $x_{R}\notin B_{R}(O)$, for $R$ sufficiently large. Furthermore, given $x\in B_{R}(O)$, estimate $(\ref{tocarry})$ implies that
	\begin{align*}
		|\nabla v|^2(x)=G_{R}(x)\leq G_{R}(x_{R})\leq\frac{C}{R}+\frac{C}{R^{2}}.
	\end{align*}
	Taking $R\rightarrow+\infty$, we obtain that $v$ is constant. This is a contradiction, which means that when $\lambda\geq0$, we could not have taken $v$ nonconstant.
	
	Now we consider the case $\lambda<0$. In addition to the properties that we have already asked $\eta$ to satisfy, we assume further that $b=R$, that
	\begin{align*}
		\eta(t)=
		\left\{ 
		\begin{array}[pos]{lll}
			1,\ t\in[0,1]\\\noalign{\smallskip}
			\frac{1+R-t}{R},\ t\in[2,1+R)\\\noalign{\smallskip}
			0,\ t\in[1+R,+\infty)
		\end{array}
		\right.
	\end{align*}
	and also that, for any $t\in(0,1+R)$,
	\begin{align}\label{condicondi}
		-\frac{2}{R}\leq\eta'\leq0\ \ \ \ \text{and}\ \ \ \ \ |\eta''|\leq C
	\end{align}
	for a universal positive constant $C$. The non smoothness of $\eta$ at $1+R$ is not a problem, once $\eta(1+R)=0$, and we are interested in $x_{R}$, that satisfy $F_{R}(x_{R})>0$.
	
	As we have already seen, when $x_{R}\in B_{O}(R)$ we have $(\ref{lamneg})$, and thus the result follows. Therefore, we may assume that $x_{R}\notin B_{O}(R)$.
	
	If $x_{R}\in B_{O}(R(1+R))- B_{O}(2R)$, once $\eta(t)=\frac{1+R-t}{R}$ and $\eta'(t)=-\frac{1}{R}$ for $t\in[2,1+R)$, we get from $(\ref{tocarry})$ that
	\begin{align}
		G_{R}(x_{R})&\leq\frac{C}{R}+\frac{C}{R^{2}}-\frac{m\lambda(1+R)}{R}+\frac{m\lambda}{R^2}r(x_{R})-\frac{m\lambda}{2R^2}r(x_{R})\nonumber\\
		&\leq\frac{C}{R}+\frac{C}{R^{2}}-m\lambda+\frac{m\lambda}{2R^2}r(x_{R})\label{case1.1}\\
		&\leq\frac{C}{R}+\frac{C}{R^{2}}-m\lambda\nonumber
	\end{align}
	If $x_{R}\in B_{O}(2R)-B_{O}(R)$, then $R\leq r(x_{R})\leq2R$. Using $\eta\leq1$ and $(\ref{condicondi})$, we get from $(\ref{tocarry})$ that
	\begin{align}
		G_{R}(x_{R})&\leq\frac{C}{R}+\frac{C}{R^{2}}-m\lambda\eta-\frac{m\lambda}{2R^2}r(x_{R})\nonumber\\
		&\leq\frac{C}{R}+\frac{C}{R^{2}}-m\lambda-\frac{m\lambda}{R}\label{case1.2}\\
		&\leq\frac{C}{R}+\frac{C}{R^{2}}-m\lambda\nonumber
	\end{align}
	
	From $(\ref{lamneg})$, $(\ref{case1.1})$ and $(\ref{case1.2})$, it follows that for any $x\in B_{O}(r)$,
	\begin{align*}
		|\nabla v|^2(x)=G_{R}(x)\leq G_{R}(x_{R})\leq\frac{C}{R}+\frac{C}{R^{2}}-m\lambda.
	\end{align*}
	Taking $R\rightarrow+\infty$, we get $|\nabla v|^2(x)\leq-m\lambda$, which finishes that proof.
\end{proof}

If we repeat the same arguments as above when $\lambda<0$ and $\mu<0$, replacing $(\ref{eqcutgrad})$ with $(\ref{eqcutexp})$, we obtain the following theorem.
\begin{theorem}\label{negneg}
	Let $M^n$ be a manifold for which there are functions $\varphi,v\in C^{\infty}(M)$ and constants $\lambda,\ \mu\in\mathbb{R}$ and $m\in(0,+\infty)$ satisfying $(\ref{eiteiS})$ and $(\ref{neweq2})$. Suppose that $\lambda<0$ and $\mu<0$. Then $\displaystyle v\leq\frac{m}{2}\ln(\lambda/\mu)$.
\end{theorem}

\subsection{The first eigenvalue of $-\Delta_{\varphi}$}\label{fujita'sargument}
Let $(M,g,e^{-\varphi}dV)$ be a weighted Riemannian manifold, where $\varphi\in C^{\infty}(M)$. Let $\Omega\subset M$ be an open bounded domain with smooth boundary. The weighted laplacian acting on $C^{\infty}(\Omega)$ is the operator $\Delta_{\varphi}$ defined by
\begin{align}\label{weightedlap}
	\Delta_{\varphi}u=\Delta u-\langle\nabla\varphi,\nabla u\rangle,\ u\in C^{\infty}(\Omega).
\end{align}
If  $u\in C^{\infty}(\Omega)$ and  $\psi\in C_{0}^{\infty}(\Omega)$, this operator satisfies the identity
\begin{align}\label{byparts}
	\int_{\Omega}\psi\Delta_{\varphi}ue^{-\varphi}dV=\int_{\Omega}u\Delta_{\varphi}\psi e^{-\varphi}dV+\int_{\partial\Omega}u\frac{\partial\psi}{\partial\nu}e^{-\varphi}dS,
\end{align}
where $e^{-\varphi}dS$ is the area element induced in $\partial\Omega$ and $\nu$ is the inward unit normal of $\partial\Omega$.

 In this section we deal with the solutions of following nonlinear problem
\begin{align}\label{mainprob}
	\left\{ 
	\begin{array}[pos]{ll}
		-\Delta_{\varphi}u=c_{1}+c_{2}e^{u},\ \text{in $\Omega$}\\
		u\geq0,\ \text{in $\overline{\Omega}$},
	\end{array}
	\right.
\end{align}
for given constants $c_{1}$ and $c_{2}$.

In what follows, denote by $\psi$ the eigenfunction associated to the smallest positive eigenvalue of the Dirichlet problem for $-\Delta_{\varphi}$, denoted by $\lambda_{1}$. Thus,
\begin{align}\label{eigenvrpob}
	\left\{ 
	\begin{array}[pos]{ll}
		-\Delta_{\varphi}\psi=\lambda_{1}\psi,\ \text{in $\Omega$}\\
		\psi=0,\ \text{in}\ \partial\Omega.
	\end{array}
	\right.
\end{align}
Let us assume that $\psi>0$ in $\Omega$ and normalize $\psi$ so that $\displaystyle\int_{\Omega}\psi e^{-\varphi}dV=1$. We also recall that, according to Hopf's Maximum Principle,
\begin{align}\label{inwardderpos}
\frac{\partial\psi}{\partial\nu}(p)>0,\ p\in\partial\Omega.
\end{align}

Below we prove the main result of this subsection, which takes advantages of the convexity of the nonlinearity in $(\ref{eigenvrpob})$.

\begin{theorem}\label{integlemma}
	Suppose that $u$ is a nonconstant solution of $(\ref{mainprob})$ so that $\displaystyle\int_{\Omega}u\psi e^{-\varphi}dV>0$.
	\begin{enumerate}
		\item If $c_{2}>0$ and $\displaystyle\lambda_{1}\inf_{\partial\Omega}{u}+c_{1}+c_{2}\geq0$, then $\lambda_{1}(\overline{\Omega})>c_{2}$.
		\item If $c_{2}\leq0$, then $\displaystyle\sup_{\partial\Omega}{u}>-\frac{c_{1}+c_{2}}{\lambda_{1}(\overline{\Omega})}$.
	\end{enumerate}
\end{theorem}
\begin{proof}
	Let us consider $\displaystyle x_{0}=\int_{\Omega}u\psi e^{-\varphi}dV$. Multiplying $(\ref{mainprob})$ by $\psi$, integrating by parts as in $(\ref{byparts})$ and using $(\ref{eigenvrpob})$ one has
	\begin{align}\label{stokesNpde}
		c_{1}+c_{2}\int_{\Omega}e^{u}\psi e^{-\varphi}dV&=\int_{\Omega}(c_{1}+c_{2}e^{u})\psi e^{-\varphi}dV\nonumber\\
		&=-\int_{\Omega}u\Delta_{\varphi}\psi e^{-\varphi}dV-\int_{\partial\Omega}u\frac{\partial\psi}{\partial\nu}e^{-\varphi}dS\\
		&=\lambda_{1}\int_{\Omega}u\psi e^{-\varphi}dV-\int_{\partial\Omega}u\frac{\partial\psi}{\partial\nu}e^{-\varphi}dS\nonumber.
	\end{align}
	Now we consider the two cases.
	\begin{enumerate}
		\item Suppose $c_{2}>0$. Using $\displaystyle\int_{\Omega}\psi e^{-\varphi}dV=1$, $\displaystyle\lambda_{1}=\int_{\partial\Omega}\frac{\partial\psi}{\partial\nu}e^{-\varphi}dV$, Jessen's inequality and $(\ref{inwardderpos})$ we get
		\begin{align*}
			\begin{split}
				c_{1}+c_{2}e^{\int_{\Omega}u\psi e^{-\varphi}dV}&\leq c_{1}+c_{2}\int_{\Omega}\psi e^{u}e^{-\varphi}dV\\
				&=\lambda_{1}\int_{\Omega}u\psi e^{-\varphi}dV-\int_{\partial\Omega}u\frac{\partial\psi}{\partial\nu}e^{-\varphi}dS\\
				&\leq\lambda_{1}\int_{\Omega}u\psi e^{-\varphi}dV-\lambda_{1}\inf_{\partial\Omega}{u},
			\end{split}
		\end{align*}
	where equality holds if and only if $u$ is constant in $\overline{\Omega}$. Since we are assuming $u$ not constant in $\overline{\Omega}$, strict inequality holds in $(\ref{comput})$, that is
	\begin{align}\label{comput}
		\begin{split}
			\lambda_{1}x_{0}-c_{2}e^{x_{0}}-c_{1}-\lambda_{1}\inf_{\partial\Omega}{u}>0.
		\end{split}
	\end{align}
	Using the inequality $e^{x}\geq x+1$, we get
	\begin{align*}
		&(c_{2}-\lambda_{1})x_{0}<-c_{2}-c_{1}-\lambda_{1}\inf_{\partial\Omega}{u}.
	\end{align*}
	Once we are assuming $x_{0}>0$, if $\displaystyle\lambda_{1}\inf_{\partial\Omega}{u}+c_{1}+c_{2}\geq0$, then we get
	\begin{align}
		\lambda_{1}(\overline{\Omega})>c_{2},
	\end{align}
	as we claimed.
		\item Suppose $c_{2}\leq0$. Using $(\ref{stokesNpde})$ and Jessen's inequality as before, we get
		\begin{align}\label{comput2}
			\begin{split}
				-\lambda_{1}x_{0}+c_{2}e^{x_{0}}+c_{1}+\lambda_{1}\sup_{\partial\Omega}{u}\geq0,
			\end{split}
		\end{align}
		where equality holds if and only if $u$ is constant in $\overline{\Omega}$. If $u$ is not constant in $\overline{\Omega}$, then strict inequality holds in $(\ref{comput2})$, which by using $e^{x}\geq x+1$, gives
		\begin{align*}
			0<(\lambda_{1}-c_{2})x_{0}<c_{2}+c_{1}+\lambda_{1}\sup_{\partial\Omega}{u},
		\end{align*}
		and once we are assuming $x_{0}>0$\, we get $\lambda_{1}\sup_{\overline{\Omega}}{u}+c_{1}+c_{2}>0$. Therefore,
		\begin{align}
			\sup_{\overline{\Omega}}{u}>-\frac{c_{1}+c_{2}}{\lambda_{1}(\overline{\Omega})}.
		\end{align}
	\end{enumerate}
\end{proof}

According to whether $\lambda_{1}(M)$ vanishes, we have the following.

\begin{corollary}\label{restriction2tosolution}
	Let $M^n$ be a complete manifold and suppose that for given constants $c_{1},\ c_{2}\in\mathbb{R}$, there is a nonnegative function $u\in C^{2}(M)$ satisfying
	\begin{align}\label{mainprob2}
			-\Delta_{\varphi}u=c_{1}+c_{2}e^{u}.
	\end{align}
	\begin{enumerate}
		\item If $\lambda_{1}(-\Delta_{\varphi})=0$, then $u$ is constant provided one of the following happens:
		\begin{enumerate}
			\item $c_{2}>0$ and $c_{1}+c_{2}\geq0$, or
			\item $c_{2}\leq0$ and $c_{1}+c_{2}<0$. 
		\end{enumerate}
		\item If $\lambda_{1}(-\Delta_{\varphi})>0$, then $u$ is constant provided ne of the following happens:
		\begin{enumerate}
			\item $c_{1}\geq0$, $c_{2}>0$ and $\lambda_{1}(-\Delta_{\varphi})\leq c_{2}$, or
			\item $c_{2}\leq0$, $c_{1}+c_{2}\leq0$ and $\displaystyle\sup_{M}{u}\leq-\frac{c_{1}+c_{2}}{\lambda_{1}(-\Delta_{\varphi})}$.
		\end{enumerate}
	\end{enumerate}
\end{corollary}
Applying this corollary to complete noncompact warped product gradient Ricci solitons we get.

\begin{theorem}\label{posiposi}
	Let $M^n$ be a manifold for which there are functions $\varphi,v\in C^{\infty}(M)$ and constants $\lambda,\ \mu\in\mathbb{R}$ and $m\in(0,+\infty)$ satisfying $(\ref{eiteiS})$ and $(\ref{neweq2})$. If $\lambda>0$, $\mu>0$ and $\displaystyle v\geq \frac{m}{2}\ln(\lambda/\mu)$, then $v$ is constant.
\end{theorem}
\begin{proof}
	By hypothesis we have $\lambda>0$, $\mu>0$ and
	\begin{align*}
		\displaystyle\inf_{M}v\geq\frac{m}{2}\ln(\lambda/\mu)=\vcentcolon\frac{m}{2}v_{0}.
	\end{align*}
	Once $v$ satisfies $(\ref{neweq2})$, the function $u=\displaystyle\frac{2}{m}v-v_{0}$ satisfies $(\ref{mainprob})$ all over $M$, for the constants
	\begin{align}\label{constants}
		\begin{split}
			&c_{1}=-2\lambda,\ \ \ \text{and}\ \ \ c_{2}=2\lambda.
		\end{split}
	\end{align}
	These choices imply that $c_{1}+c_{2}=0$, and thus
	\begin{align*}
		\displaystyle\lambda_{1}\inf_{p\in\overline{\Omega}}{u(p)}+c_{1}+c_{2}\geq0.
	\end{align*}
	
	On the other hand, it follows from $(\ref{eiteiS})$ that $Ric_{\varphi}\geq\lambda g$, from where the $\varphi$-volume of $M$ is finite \cite{morgan} (see also \cite[Theorem 4.1]{weywyl}) and then, $\lambda_{1}(-\Delta_{\varphi})=0$ \cite[Proposition 3.1]{setti}. It follows from Corollary \ref{restriction2tosolution} that $v$ is constant, as claimed.
\end{proof}

\subsection{Applying the weak Maximum Principle at Infinity}\label{wmpatinf}

Let $M^n$ be a manifold for which there are functions $\varphi,v\in C^{\infty}(M)$ and constants $\lambda\in\mathbb{R}$, $\mu\leq0$ and $m\in(0,+\infty)$ satisfying $(\ref{eiteiS})$. In particular,
\begin{align*}
	Ric_{\varphi}\geq\lambda g.
\end{align*}
It thus follows from \cite[Theorem 4.1]{weywyl} that the volume of the geodesic balls with respect to $e^{-\varphi}dV$ grow in a way to assure the weak Maximum Principle at Infinity to hold on $M$. See \cite[Theorem 9]{pirise2} for more details. Using such a maximum principle and Theorem \ref{gradbound}, we generalize to expanding warped product gradient Ricci solitons a rigidity proved to expanding $m$-quasi-Einstein manifolds \cite[Theorem 5]{barros}.

\begin{theorem}\label{negneg2}
	Let $M^n$ be a complete manifold for which there are functions $\varphi,v\in C^{\infty}(M)$ and constants $\lambda,\ \mu\in\mathbb{R}$ and $m\in(0,+\infty)$ satisfying $(\ref{eiteiS})$ and $(\ref{neweq2})$. Suppose that $\lambda<0$ and $\mu<0$. If $\displaystyle\inf_{M}v>-\infty$, then
	\begin{align*}
		\displaystyle\sup_{M}|\nabla v|^2\leq-m\lambda+2m\mu e^{\frac{2}{m}\inf_{M}v}.
	\end{align*}
	In particular, if
	\begin{align}\label{boundnegat}
		\displaystyle v\geq\frac{m}{2}\ln\left(\frac{\lambda}{2\mu}\right),
	\end{align}
	then $v$ is constant.
\end{theorem}
\begin{proof}
	We know from Theorem \ref{gradbound} that $\displaystyle\sup_{M}|\nabla v|^2<-m\lambda$. Once the weak Maximum Principle at Infinity holds true in $M$, there is a sequence $(p_{k})$ in $M$ satisfying
	\begin{align*}
		(\Delta_{\varphi}|\nabla v|^2)(p_{k})\leq\frac{1}{k}\ \ \ \ \text{and}\ \ \ \ |\nabla v|^2(p_{k})\geq\sup_{M}{|\nabla v|^2}-\frac{1}{k}.
	\end{align*}
	Suppose that $\displaystyle\inf_{M}{v}>-\infty$. Using $(\ref{gradeq})$ we get
	\begin{align*}
		\frac{1}{k}\geq(\Delta_{\varphi}|\nabla v|^2)(p_{k})\geq2(\lambda-2\mu e^{\frac{2}{m}\inf_{M}v})|\nabla v|^2(p_{k})+\frac{2}{m}|\nabla v|^4(p_{k}).
	\end{align*}
	Taking $k\rightarrow+\infty$, we get
	\begin{align*}
		\sup_{M}{|\nabla v|^2}\leq-m\lambda+2m\mu e^{\frac{2}{m}\inf_{M}v},
	\end{align*}
	as claimed. On the other hand, if we assume $(\ref{boundnegat})$, then $\sup_{M}{|\nabla v|^2}=0$, and then $v$ is constant.	
\end{proof}

Now we apply the weak Minimum Principle at Infinity to equation $(\ref{neweq2})$. The author do not know whether this is already known	 for $m$-quasi-Einstein manifolds.

\begin{theorem}\label{zeronegtv}
	Let $M^n$ be a complete manifold for which there are functions $\varphi,v\in C^{\infty}(M)$ and constants $\lambda,\ \mu\in\mathbb{R}$ and $m\in(0,+\infty)$ satisfying $(\ref{eiteiS})$ and $(\ref{neweq2})$. Suppose that $\displaystyle\inf_{M}{v}>-\infty$. If either,
	\begin{enumerate}
		\item $\lambda=0$ and $\mu>0$, or
		\item $\lambda<0$ and $\mu\geq0$,
	\end{enumerate}
then $v$ is constant.
\end{theorem}
\begin{proof}
	Along the proof we assume that $\lambda\leq0$, $\mu\geq0$ and $\lambda^2+\mu^2\neq0$. Suppose that $\displaystyle\inf_{p\in M}{v(p)}>-\infty$. Once the weak Minimum Principle at Infinity holds true in $M$, there is a sequence $(p_{k})$ in $M$ so that
	\begin{align*}
		(\Delta_{\varphi}v)(p_{k})\geq-\frac{1}{k}\ \ \ \ \text{and}\ \ \ \ v(p_{k})\leq\inf_{M}{v}+\frac{1}{k}.
	\end{align*}
	Using equation $(\ref{neweq2})$ we get,
	\begin{align*}
		-\frac{1}{k}\leq m\lambda-m\mu e^{\frac{2}{m}v(p_{k})}.
	\end{align*}
	As $v(p_{k})$ converges to $\displaystyle\inf_{M}{v}$, we conclude that
	\begin{align*}
		0\leq m \lambda-m\mu e^{\frac{2}{m}\inf_{M}{v}}<0,
	\end{align*}
	what is a contradiction.
\end{proof}
\subsection{Proof of the main theorems}\label{howtoprove}

In order to prove of Theorem \ref{noncompacshrinking}, Theorem \ref{noncompacsteady} and Theorem \ref{noncompacexpanding}, we use the change of dependent coordinates $(\ref{change})$ and apply the theorems of the previous subsections.

\begin{proof}[{\bf Proof of Theorem \ref{noncompacshrinking}}]
	Suppose that $h$ is not constant and consider the change of coordinates of Proposition \ref{gomariprop}. Now apply Theorem \ref{gradbound} to conclude that $\mu>0$. To prove that under \ref{item02} the function $h$ must be constant, apply Theorem \ref{posiposi} to get a contradiction. This finishes the proof.
\end{proof}

\begin{proof}[{\bf Proof of Theorem \ref{noncompacsteady}}]
	Suppose that $h$ is not constant and consider the change of coordinates of Proposition \ref{gomariprop}. Now we use Theorem \ref{gradbound} and Theorem \ref{zeronegtv} to obtain item \ref{item1} and item \ref{item2}, respectively.	
\end{proof}

\begin{proof}[{\bf Proof of Theorem \ref{noncompacexpanding}}]
	Suppose that $h$ is not constant and consider the change of coordinates of Proposition \ref{gomariprop}. Now apply Theorem \ref{gradbound}, Theorem \ref{negneg}, Theorem \ref{negneg2} and Theorem \ref{zeronegtv} to obtain item \ref{negitem1}, item \ref{negitem2}, item \ref{negitem3} and item \ref{negitem4}, respectively.
\end{proof}



\vspace{1cm}

\hspace{-.8cm} Universidade Federal do Pará - UFPA\\
Faculdade de Matem\'{a}tica, 66075-110 Be\-l\'{e}m, Par\'{a} - PA, Brazil.\\
valterborges@ufpa.br

\end{document}